\documentclass[reqno, 11 pt]{amsart}

\usepackage[english]{babel}
\usepackage{amsmath}
\usepackage{amsthm}%Entorno proof%
\usepackage{amsfonts}
\usepackage{amssymb}
\usepackage{anysize}
\usepackage{hyperref}
\usepackage{appendix}
\usepackage{mathtools}
\usepackage{graphicx}
\usepackage{xcolor}
\usepackage{amsmath,esint,bigints, amssymb, mathrsfs, amsthm,mathbbol,upgreek,exscale}

\usepackage{ stmaryrd }

\newtheorem{defi}{Definition}[section]%Indicar si la numeracion sigue secciones, subsecciones, ...
\newtheorem{lem}[defi]{Lemma}
\newtheorem{theo}[defi]{Theorem}

\newtheorem{pro}[defi]{Proposition}

\newtheorem{rem}[defi]{Remark}

\DeclareMathOperator{\R}{\mathbb{R}}

\DeclareMathOperator{\T}{\mathbb{T}}

\newcommand{\norm}[1]{\left\lVert#1\right\rVert} 
\DeclarePairedDelimiter\abs{\lvert}{\rvert}% 

\allowdisplaybreaks %Salto de página en align-enviroment%

\title[Traveling gravity-capillary waves with odd viscosity]{Traveling gravity-capillary waves with odd viscosity}

\author[D. Alonso-Or\'an]{Diego Alonso-Or\'an}
\address{ Departamento de An\'alisis Matem\'atico y Instituto de Matem\'aticas y Aplicaciones (IMAULL), Uni-
versidad de La Laguna C/. Astrof\'isico Francisco S\'anchez s/n, 38200 - La Laguna, Spain}
\email{dalonsoo@ull.edu.es}

\author[C. Garc\'ia]{Claudia Garc\'ia}
\address{ Departamento de Matem\'atica Aplicada \& Research Unit ``Modeling Nature'' (MNat), Facultad de Ciencias, Universidad de Granada, 18071 Granada, Spain}
\email{claudiagarcia@ugr.es}

\author[R. Granero-Belinch\'on]{Rafael Granero-Belinch\'on}
\address{ Departamento de Matem\'aticas, Estad\'istica y Computaci\'on, Universidad de Cantabria. Avda. Los Castros s/n, Santander, Spain.}
\email{rafael.granero@unican.es}

{\thanks{}}

\begin{document}

\date{\today}

\begin{abstract}
In this note, we study the existence of traveling waves of a surface model in a non-newtonian fluid with odd viscosity. The proof relies on nonlinear bifurcation techniques.
\end{abstract}

\maketitle

\tableofcontents

\section{Introduction and main result}
The motion of an incompressible fluid is described by the following non-linear system
\begin{align}\label{fluids:system}
\rho \left(u_{t}+(u\cdot\nabla)u\right)&= \nabla\cdot \mathcal{T}, \nonumber  \\
\rho_{t}+\nabla\cdot (u\rho)&=0,\nonumber  \\
\nabla\cdot u&=0,
\end{align}
where $u(x,t):\Omega\times [0,T]\to \mathbb{R}^{d}, \quad \rho(x,t): \Omega\times [0,T]\to \mathbb{R}$ represent the the velocity field and density of the fluid respectively and $d=2,3$. Moreover, $\mathcal{T}$ denotes the stress tensor that varies depending the different properties of the fluid. For inviscid and newtonian fluids, the stress tensor is given by 
\[ \mathcal{T}^{i}_{j}=-p\delta_{j}^{i},\] and system \eqref{fluids:system} reduces to the well-known Euler equations. For viscous newtonian fluids, the stress tensor has even symmetry and takes the form
\[ \mathcal{T}^{i}_{j}=-p\delta_{j}^{i}+\left( \partial_{x_{j}}u^{i}+\partial_{x_{i}}u^{j} \right).\]
However, in stark contrast to the classical scenario for viscous or inviscid newtonian fluids, there are classes of fluids with broken microscopic time-reversal symmetry and parity, namely quantum fluid (magnetized plasmas or electron fluids) or classical fluid systems (polyatomic gases). We refer the interested reader to \cite{AvronSeiler95} for a more detailed explanation and discussion. 
In two-dimensional fluid systems where microscopic time reversal and parity are violated, the viscosity tensor includes a skew-symmetric component often referred to as odd viscosity given by
\begin{equation}\label{stress:tensorodd}
\mathcal{T}^{i}_{j}=-p\delta_{j}^{i}+\left( \partial_{x_{i}}\left(u^{j}\right)^{\perp}+\left(\partial_{x_{i}}\right)^{\perp}u^{j} \right).
\end{equation}
Although in three dimensions, terms in the viscosity tensor with odd symmetry were known in the context of anisotropic fluids \cite{Land-Lif_Fluids}, Avron noticed that in two dimensions, odd viscosity and isotropy can hold at the same time, \cite{Avron}. Although the recently increasing interest of the mathematical and physical community in fluids with odd viscosity effects, there are not so many mathematical works considering this setting. 

Recently, in \cite{FGS}, the authors establish a well-posedness theory in Sobolev spaces for a system of incompressible non-homogeneous fluids with odd viscosity given by \eqref{stress:tensorodd}. A well-posedness theory in Besov spaces was later proved in \cite{FV}. Remarkably, in this last paper the authors manage to prove the the solution is asymptotically global in the sense that the lifespan grows as the density tends to homogeneity. \medskip

Recently, in \cite{Granero-Ortega21}, the authors obtained three new models for capillary–gravity surface waves with odd viscosity through a multi-scale expansion in the steepness of the wave. The multi-scale expansion approach (cf. \cite{ADG,CGSW}) reduces the full system to a cascade of linear equations which can be closed up to some order of precision. The derived models in \cite{Granero-Ortega21} consider effects of both gravity and surface tension forces generalizing those in \cite{Abanov2010,abanov2019free}. One of the asymptotic models studies the unidirectional surface waves, given by the dispersive equation
\begin{align}\label{dispersive:eq}
2f_t +\alpha_0\Lambda [f_t]=&\frac{1}{\varepsilon}\left\{f_{x}+\mathcal{H}[f]+(\alpha_0-\beta)\mathcal{H}[f_{x x}]\right\}\nonumber \\
&+\mathcal{H}[(\Lambda f)^2]-\llbracket\mathcal{H},f\rrbracket[\Lambda f]+(\alpha_0-\beta)\llbracket\mathcal{H},f\rrbracket\Lambda^3 f,
\end{align}
where $f:[0,T]\times\T\rightarrow \R$. Here $\varepsilon$ is known as the steepness parameter and measures the ratio between the amplitude and the wavelength of the wave, $\alpha_0$ is linked to the Reynolds numbers and represents the ratio between gravity and odd viscosity forces. In this paper we assume this parameter to be strictly positive. Finally, $\beta$ is the Bond number comparing the gravity and capillary forces. Notice that \eqref{dispersive:eq} conserves the total mass of water for periodic domains and for waves that decay fast enough at infinity. Besides the derivation of the model \eqref{dispersive:eq} the authors in \cite{Granero-Ortega21} showed the locally well-posedness in $H^{3}(\R)$ when the odd Reynolds number $\alpha_{0}$ is strictly positive regardless and no assumption on the value of the Bond number $\beta$. Furthermore, for $0 < \alpha_{0}= \beta$, the problem admits a distributional solution in $H^{1.5}(\R)$. \medskip

The main result provided is this manuscript shows the existence of traveling waves for equation \eqref{dispersive:eq} and reads as follows:

\begin{theo}\label{main:theorem}
For $0<\alpha_{0}\neq \beta$, and for any $m\geq 1$, there exists a one dimensional curve $s\mapsto (c_s, \varphi_s)$, with $s\in I$, such that
$$
f_0(x)=\varphi_s(x),
$$
is a m-fold traveling wave solution to \eqref{dispersive:eq} with constant speed $c_s$.
\end{theo}

\begin{rem}
For \(0 < \alpha_0 = \beta\), one can mimic the proof of Theorem \ref{main:theorem} and recover the same result. Notice that for \(0 < \alpha_0 = \beta\), the singular commutator term \(\llbracket \mathcal{H}, f \rrbracket \Lambda^3 f\) is not present, and there is no need to invoke the commutator Lemma \ref{lemma:comutator} to show the analogue of Proposition \ref{well:define1}. 

\end{rem}

The literature regarding the study of permanent progressive waves, as solitary and traveling waves, is a key area of interest. These waves, also known as steady waves, propagate without changing their shape over time. Given the significant complexity of the classical water wave problem, numerous approximate models have been studied since the early years. These models are formally derived through various scaling limits. Perhaps the canonical example is the so called Korteweg-de Vries (KdV) equation
\[ u_{t}+3(u^2)_{x}+u_{xxx}=0,\]
to model propagation of surface water waves with small amplitudes and long wavelengths in a channel, \cite{Boussinesq,KortewegVries}. The KdV equation includes the essential effects of nonlinearity and dispersion. The mathematical theory for the KdV equation is well-known, featuring a theory of well-posedness and a thorough understanding of the stability properties of solitary and traveling waves, \cite{Benjamin,Bona,GSS,Kato2}. Similarly, there have been other successful models where the existence of traveling waves have been extensively studied such as the Fornberg-Whitam equation \cite{Fornberg}
\[ u_{xxt}-u_{t}+\frac{9}{2}u_{x}u_{xx}+\frac{3}{2}uu_{xxx}-\frac{3}{2}uu_{x}+u_{x}=0,\]
 or the Camassa-Holm equation \cite{CH,Fuchs}
\[u_{t}-u_{txx}+3uu_{x}+2u_{x}=2u_{x}u_{xx}+uu_{xxx}.\]
For the former, traveling wave solutions of kink-like and antikink-like wave solutions were recently investigated in \cite{ZhouTian} and the references therein. The latter, has been deeply analyze and all types of traveling waves solutions are classified such as peakons, cuspons, stumpons, and composite waves, cf. \cite{FerreriaKraenkel,Lenells1,Lenells2}. \medskip

To the best of the author's knowledge, Theorem \ref{main:theorem} seems to be the first rigorous result regarding the existence of traveling waves solutions for fluids with odd viscosity effects.

\subsubsection*{Plan of the paper}
In Section \ref{sec:2}, we present the notation used throughout the article as well as some auxiliary results. In particular, we provide a commutator estimate for the Hilbert transform in H\"older spaces and recall basic tools in bifurcation theory. In Section \ref{sec:3} we introduce the formulation of the problem as well as the function spaces that will be used in order to implement the Crandall-Rabinowitz theorem. In Subsection \ref{sec:31} we study the spectral properties of the linearized operator and check such linear operator is a Fredholm operator of zero index. Finally, we also study the kernel and the range to verify the transversality condition. In Subsection \ref{sec:32}, gathering the different results provided previously, we invoke the Crandall-Rabinowitz theorem to show the proof of the Theorem \ref{main:theorem}.

\section{Notation and auxiliary results}\label{sec:2}
For a function $f\in \T$ with values in $\R$, we define the H\"older norms as 
\begin{align*}
\norm{f}_{C^0(\T)}&=\sup_{x\in \mathbb{\T}^1} \abs{f}, \; \norm{f}_{C^k(\T)}=\norm{f}_{C^0(\T)}+\sum_{\ell=1}^k \norm{\partial_x^\ell f}_{C^0(\T)},\; k\in \mathbb{N},\\
\norm{f}_{C^{\alpha}(\T)}&=\norm{f}_{C^0(\T)}+\sup_{x_1,x_2\in (\T)} \frac{\abs{f(x_1)-f(x_2)}}{\abs{x_1-x_2}^{\alpha}}, \; 0<\alpha<1,\\
\norm{f}_{C^{k,\alpha}}&=\norm{f}_{C^{k-1}(\T)}+\norm{\partial_x^k f}_{C^\alpha(\T)}, \; k\in \mathbb{N},\; 0<\alpha<1.
\end{align*}
The Banach space of continuous functions for which the above norms are finite will be denoted $C^{k}(\T;\mathbb{R})$ and $C^{k,\gamma}(\T;\mathbb{R})$. The linear operator $\mathcal{H}$ refers to the Hilbert transform in the periodic setting and is given by
\begin{equation}\label{def:Hilbert}
\mathcal{H}[f](x)=\frac{1}{2\pi} p.v. \int_{-\pi}^{\pi} \frac{f(y)}{\tan\left(\frac{x-y}{2}\right)} dy,
\end{equation}
and $\Lambda=\mathcal{H}\partial_{x}$ the Zygmund operator is defined as
\begin{equation}\label{def:Lambda}
\Lambda[f](x)=\frac{1}{4\pi}p.v. \int_{-\pi}^{\pi}\frac{f(x)-f(x-y)}{\sin^2(\frac{y}{2})}dy.
\end{equation}
Moreover, given an operator $\mathsf{T}$, we define the commutator as  $\llbracket\mathsf{T},f\rrbracket [g]=\mathsf{T}(fg)-f\mathsf{T}(g)$. \medskip

We will denote with $C$ a positive generic constant that depends only on fixed parameters. Note also that this constant might differ from line to line. \medskip

Next, we show a commutator estimate for the Hilbert transform $\mathcal{H}$ in H\"older spaces. Similar commutator estimates to the one provided in this article can be found in  \cite[Lemma B.1]{Constantin-Varvaruca11} in the context of water waves or \cite[Lemma 2.2]{GarciaJuarezetal24} for the interface Stokes flow problem. However, to the best of the authors knowledge, the estimate provided here does not follow from the previous results.

\begin{lem}\label{lemma:comutator}
Let $\alpha\in (0,1)$, $a\in C^{2,\alpha}(\T), \ b\in C^{1,\alpha}(\T)$. Then, we have that
\begin{equation}\label{est:commutator}
\norm{\llbracket\mathcal{H},a \rrbracket [b']}_{C^{1,\alpha}(\T)}\leq C\norm{a}_{C^{2,\alpha}(\T)}\norm{b}_{C^{1,\alpha}(\T)}.
\end{equation}
\end{lem}

\begin{proof}
In order to ease the notation, we denote by $\Theta(x):= \llbracket\mathcal{H},a\rrbracket [b'](x)$. Using the definition of $\mathcal{H}$ in \eqref{def:Hilbert} we readily check that
\begin{align*}
\Theta(x)=\frac{1}{2\pi}\textnormal{p.v.}\int_{-\pi}^{\pi} \frac{\left(a(y)-a(x)\right)}{\tan\left(\frac{x-y}{2}\right)}b'(y)dy=-\frac{1}{2\pi}\textnormal{p.v.}\int_{-\pi}^{\pi} \frac{\left(a(y)-a(x)\right)}{\tan\left(\frac{x-y}{2}\right)}\left(b(x)-b(y)\right)'dy.
\end{align*}
Using integration by parts, we have that
\[ \Theta(x)=\frac{1}{2\pi}\textnormal{p.v.}\int_{-\pi}^{\pi} \left( \frac{a'(y)}{\tan(\frac{x-y}{2})}+\frac{\left(a(y)-a(x)\right)}{2\sin^{2}(\frac{x-y}{2})}\right)\left(b(x)-b(y)\right)dy. \]
Using the change of variable $\tilde{y}=x-y$ and further manipulation yields
\[ \Theta(x)=\frac{1}{4\pi}\textnormal{p.v.}\int_{-\pi}^{\pi} \left( \frac{2a'(x-\tilde{y})\sin(\frac{\tilde{y}}{2})\cos(\frac{\tilde{y}}{2})-\left(a(x)-a(x-\tilde{y})\right)}{\sin^{2}(\frac{\tilde{y}}{2})}\right)\left(b(x)-b(x-\tilde{y})\right)d\tilde{y}. \]
From now on, we will just write $y$ instead of $\tilde{y}$. Taking into account that 
\begin{equation}\label{simpli}
 \frac{1}{\sin^{2}(y)}\leq g(|y|)\frac{1}{y^{2}}, \quad 0<|y|<\pi, \quad \abs{\sin(y)-y}\leq \frac{1}{6}|y|^{3}, \quad y\in\mathbb{R},
 \end{equation}
where $g(|y|): (0,\pi)\to [0,\infty)$ is a bounded function, we find that
\[ |\Theta(x)|\leq C \textnormal{p.v.}\int_{-\pi}^{\pi} \left( \frac{\norm{a}_{C^{1}(\T)}|y|+\abs{\left(a(x)-a(x-y)\right)}}{y^{2}}\right)\abs{\left(b(x)-b(x-y)\right)} g(|y|) dy+ \textrm{l.o.t} \]
Therefore, since $g(|y|)$ is a bounded function and
\begin{equation}\label{simpli2}
|a(x)-a(x-y)|\leq |y| \norm{a}_{C^{1}(\T)}, \quad \abs{b(x)-b(x-y)}\leq |y|^{\alpha}\norm{b}_{C^{\alpha}(\T)},
\end{equation}
we have that
\begin{equation}\label{estimate:Linfty}
\norm{\Theta}_{C^0(\T)}\leq C \norm{a}_{C^{1}(\T)}\norm{b}_{C^{\alpha}(\T)} \textnormal{p.v.}\int_{-\pi}^{\pi}\frac{1}{|y|^{1-\alpha}}g(|y|) dy \leq C\norm{a}_{C^{1}(\T)}\norm{b}_{C^{\alpha}(\T)}.  
 \end{equation}
To compute the higher order norm, we first notice that
 \begin{align*}
 \Theta'(x)&=\frac{1}{4\pi}\textnormal{p.v.}\int_{-\pi}^{\pi} \left( \frac{2a''(x-\tilde{y})\sin(\frac{\tilde{y}}{2})\cos(\frac{\tilde{y}}{2})-\left(a'(x)-a'(x-\tilde{y})\right)}{\sin^{2}(\frac{\tilde{y}}{2})}\right)\left(b(x)-b(x-\tilde{y})\right)d\tilde{y} \nonumber \\
&\quad + \frac{1}{4\pi}\textnormal{p.v.}\int_{-\pi}^{\pi} \left( \frac{2a'(x-\tilde{y})\sin(\frac{\tilde{y}}{2})\cos(\frac{\tilde{y}}{2})-\left(a(x)-a(x-\tilde{y})\right)}{\sin^{2}(\frac{\tilde{y}}{2})}\right)\left(b'(x)-b'(x-\tilde{y})\right)d\tilde{y}.
 \end{align*}
To compute the $C^{\alpha}$ norm for $\Theta'(x)$, it is convenient to introduce the difference notation
\[ \Delta_{h}f:= f(x+h)-f(x+h-y), \quad  \Delta_{y}f:=f(x)-f(x-y).\]
Hence, calculating the H\"older difference for $h>0$ yields
\begin{align}\label{diff:hold}
\Theta'(x+h)-\Theta'(x)&=\frac{1}{4\pi}\textnormal{p.v.}\int_{-\pi}^{\pi} \frac{1}{\sin^{2}(\frac{y}{2})}  \bigg[ \Delta_{h} b \left(2a''(x+h-y)\sin(\frac{y}{2})\cos(\frac{y}{2})-\Delta_{h}a' \right) \nonumber \\
&\hspace{3.2cm}- \Delta_{y}b\left(2a''(x-y) \sin(\frac{y}{2})\cos(\frac{y}{2})-\Delta_{y}a'\right)\bigg] dy \nonumber \\
&\quad + \frac{1}{4\pi}\textnormal{p.v.}\int_{-\pi}^{\pi} \frac{1}{\sin^{2}(\frac{y}{2})}  \bigg[ \Delta_{h} b' \left(2a'(x+h-y)\sin(\frac{y}{2})\cos(\frac{y}{2})-\Delta_{h}a \right) \nonumber \\
&\hspace{2.5cm}- \Delta_{y}b'\left(2a'(x-y) \sin(\frac{y}{2})\cos(\frac{y}{2})-\Delta_{y}a\right)\bigg] dy=\mathsf{I}_{1}+\mathsf{I}_{2}.
\end{align} 
Adding and subtracting $\Delta_{h} b \left(2a''(x-y) \sin(\frac{y}{2})\cos(\frac{y}{2})-\Delta_{y}a'\right)$ in $\mathsf{I}_{1}$ we find that $\mathsf{I}_{1}=\mathsf{I}_{11}+\mathsf{I}_{12}$ 
where
\begin{align*}
\mathsf{I}_{11}&=\frac{1}{4\pi}\textnormal{p.v.}\int_{-\pi}^{\pi} \frac{1}{\sin^{2}(\frac{y}{2})}\Delta_{h} b   \bigg[ \left(2a''(x+h-y)\sin(\frac{y}{2})\cos(\frac{y}{2})-\Delta_{h}a'\right) \\
&\hspace{7cm} -\left(2a''(x-y) \sin(\frac{y}{2})\cos(\frac{y}{2})-\Delta_{y}a'\right)\bigg] dy, \\
\mathsf{I}_{12}&=\frac{1}{4\pi}\textnormal{p.v.}\int_{-\pi}^{\pi} \frac{1}{\sin^{2}(\frac{y}{2})} \bigg[2a''(x-y) \sin(\frac{y}{2})\cos(\frac{y}{2})-\Delta_{y}a'\bigg] \left(\Delta_{h} b -\Delta_{y}b \right) dy.
\end{align*}
Similarly, adding and subtracting $\Delta_{h} b' \left(2a'(x-y) \sin(\frac{y}{2})\cos(\frac{y}{2})-\Delta_{y}a\right)$ in $\mathsf{I}_{2}$ we find that $\mathsf{I}_{1}=\mathsf{I}_{21}+\mathsf{I}_{22}$ 
where
\begin{align*}
\mathsf{I}_{21}&=\frac{1}{4\pi}\textnormal{p.v.}\int_{-\pi}^{\pi} \frac{1}{\sin^{2}(\frac{y}{2})}\Delta_{h} b'   \bigg[ \left(2a'(x+h-y)\sin(\frac{y}{2})\cos(\frac{y}{2})-\Delta_{h}a\right) \\
&\hspace{7cm} -\left(2a'(x-y) \sin(\frac{y}{2})\cos(\frac{y}{2})-\Delta_{y}a\right)\bigg] dy, \\
\mathsf{I}_{22}&=\frac{1}{4\pi}\textnormal{p.v.}\int_{-\pi}^{\pi} \frac{1}{\sin^{2}(\frac{y}{2})} \bigg[2a'(x-y) \sin(\frac{y}{2})\cos(\frac{y}{2})-\Delta_{y}a\bigg] \left(\Delta_{h} b' -\Delta_{y}b' \right) dy.
\end{align*}
Thus
\[
\Theta'(x+h)-\Theta'(x)= \mathsf{I}_{11}+\mathsf{I}_{12}+\mathsf{I}_{21}+\mathsf{I}_{22}.
\]
We will just bound the first two integrals $\mathsf{I}_{11}, \mathsf{I}_{12}$. The remaining terms $\mathsf{I}_{21}, \mathsf{I}_{22}$ can be estimated in a similar fashion. Let us start with $\mathsf{I}_{12}$. We write
\begin{align*}
\mathsf{I}_{12}&=\frac{1}{4\pi}\textnormal{p.v.}\int_{-\pi}^{\pi} \frac{1}{\sin^{2}(\frac{y}{2})} \bigg[a''(x-y) y\cos(\frac{y}{2})-\Delta_{y}a'\bigg] \left(\Delta_{h} b -\Delta_{y}b \right) dy  \\
&\quad -\frac{1}{4\pi}\textnormal{p.v.}\int_{-\pi}^{\pi} \frac{1}{\sin^{2}(\frac{y}{2})} \bigg[a''(x-y)\left( \sin(\frac{y}{2})-\frac{y}{2}\right)\cos(\frac{y}{2})-\Delta_{y}a'\bigg] \left(\Delta_{h} b -\Delta_{y}b \right) dy.
\end{align*}
We just bound the first integral above, since using \eqref{simpli} the later is even easier to control and the same estimate follows. To that purpose we first notice that
\[ 
\Delta_{y}a'=y\int_{0}^{1} a''(\lambda x+(1-\lambda)(x-y)) \ d\lambda,
\]
and hence
\begin{align}\label{bound:extra}
 \abs{a''(x-y) y\cos(\frac{y}{2})-\Delta_{y}a'}&\leq \abs{(1-\cos(\frac{y}{2}))}\abs{y}\int_{0}^{1} \abs{\left(a''(\lambda x+(1-\lambda)(x-y)-a''(x-y)\right)} \ d\lambda \nonumber \\
 &\leq C \abs{y}^{1+\alpha}\norm{a'}_{C^{1,\alpha}(\T)}\leq C \abs{y}^{1+\alpha}\norm{a}_{C^{2,\alpha}(\T)}.
 \end{align}
Thus, combining \eqref{simpli},\eqref{simpli2} together with \eqref{bound:extra} we infer that
\begin{equation}\label{estimate:II}
\mathsf{I}_{12}\leq C\norm{a}_{C^{2,\alpha}(\T)}\norm{b}_{C^{\alpha}(\T)} h^{\alpha}\textnormal{p.v.}\int_{-\pi}^{\pi} \frac{1}{|y|^{1-\alpha}} dy \leq C\norm{a}_{C^{2,\alpha}(\T)}\norm{b}_{C^{\alpha}(\T)} h^{\alpha}.
\end{equation}
Next, let us bound $\mathsf{I}_{11}$. We split the integral as
\begin{align*}
 \mathsf{I}_{11}= \frac{1}{4\pi}\textnormal{p.v.}\int_{-2|h|}^{2|h|} [\ldots] \ dy + \frac{1}{4\pi}\textnormal{p.v.}\int_{(-\pi,-2|h|)\cup (2|h|,\pi)}  [\ldots] \ dy:=\mathsf{J}_{in}+\mathsf{J}_{out}.
 \end{align*}
Again we can add and subtract $\frac{y}{2}$ and write
\begin{align*}
\mathsf{J}_{in}&=\frac{1}{4\pi}\textnormal{p.v.}\int_{-2|h|}^{2|h|} \frac{1}{\sin^{2}(\frac{y}{2})}\Delta_{h} b   \bigg[ \left(a''(x+h-y)y\cos(\frac{y}{2})-\Delta_{h}a'\right)-\left(a''(x-y) y\cos(\frac{y}{2})-\Delta_{y}a'\right)\bigg]  dy\\
&\quad +\frac{1}{4\pi}\textnormal{p.v.}\int_{-2|h|}^{2|h|} \frac{1}{\sin^{2}(\frac{y}{2})}\Delta_{h} b   \bigg[ \left(a''(x+h-y) \left(\sin(\frac{y}{2})-y\right)\cos(\frac{y}{2})-\Delta_{h}a'\right)\\
&\hspace{7cm}-\left(a''(x-y)  \left(\sin(\frac{y}{2})-y\right)\cos(\frac{y}{2})-\Delta_{y}a'\right) \bigg] dy.
\end{align*}
Similarly as before, using \eqref{simpli} the second integral above less singular easier to control. Therefore, using \eqref{simpli}-\eqref{simpli2} and writing again
\begin{equation}\label{truqui}
\Delta_{y}a'=y\int_{0}^{1} a''(\lambda x+(1-\lambda)(x-y)) \ d\lambda, \quad \Delta_{h}a'=y\int_{0}^{1} a''(\lambda (x+h)+(1-\lambda)(x+h-y)) \ d\lambda,
\end{equation}
we readily check that
\begin{align}\label{estimate:in}
\mathsf{J}_{in}\leq C\norm{b}_{C^{\alpha}(\T)}\norm{a}_{C^{2}(\T)} \textnormal{p.v.}\int_{-2|h|}^{2|h|} \frac{|y|^{1+\alpha}}{y^{2}} dy &\leq  C\norm{b}_{C^{\alpha}(\T)}\norm{a}_{C^{2}(\T)}\textnormal{p.v.}\int_{-2|h|}^{2|h|} \frac{1}{y^{1-\alpha}} dy \nonumber \\
&\leq  C\norm{b}_{C^{\alpha}(\T)}\norm{a}_{C^{2}(\T)}h^{\alpha}.
\end{align}
To bound the outer part $\mathsf{J}_{out}$, we perform a similar estimate. First, we add and subtract $\frac{y}{2}$ and write
\begin{align*}
\mathsf{J}_{out}&=\frac{1}{4\pi}\int_{(-\pi,-2|h|)\cup (2|h|,\pi)} \frac{1}{\sin^{2}(\frac{y}{2})}\Delta_{h} b   \bigg[ \left(a''(x+h-y)y\cos(\frac{y}{2})-\Delta_{h}a'\right) \\
&\hspace{8cm}-\left(a''(x-y) y\cos(\frac{y}{2})-\Delta_{y}a'\right)\bigg]  dy\\
&\quad +\frac{1}{4\pi}\int_{(-\pi,-2|h|)\cup (2|h|,\pi)} \frac{1}{\sin^{2}(\frac{y}{2})}\Delta_{h} b   \bigg[ \left(a''(x+h-y) \left(\sin(\frac{y}{2})-y\right)\cos(\frac{y}{2})-\Delta_{h}a'\right)\\
&\hspace{7cm}-\left(a''(x-y)  \left(\sin(\frac{y}{2})-y\right)\cos(\frac{y}{2})-\Delta_{y}a'\right) \bigg] dy.
\end{align*}
Once again we just bound the former integral, being the later less singular. On the one hand, we find that
\begin{align}\label{cancel1}
\abs{a''(x+h-y)y\cos(\frac{y}{2})-a''(x-y) y\cos(\frac{y}{2})}\leq \norm{a}_{C^{2,\alpha}(\T)} |y| h^{\alpha}.
\end{align}
On the other hand, recalling \eqref{truqui}, we can write
\[
\Delta_{y}a'-\Delta_{h}a'=y \int_{0}^{1}\left(a''(\lambda x+(1-\lambda)(x-y))-a''(\lambda (x+h)+(1-\lambda)(x+h-y)) \right) d\lambda,
\]
and hence
\begin{equation}\label{cancel2}
 \abs{\Delta_{y}a'-\Delta_{h}a'}\leq |y|  \norm{a}_{C^{2,\alpha}(\T)}h^{\alpha}.
 \end{equation}
 Combining \eqref{cancel1}-\eqref{cancel2} together with \eqref{simpli}-\eqref{simpli2} we conclude that
 \begin{align}\label{estimate:out}
\mathsf{J}_{out}&\leq C \norm{a}_{C^{2,\alpha}(\T)}\norm{b}_{C^{\alpha}(\T)} h^{\alpha} \int_{(-\pi,-2|h|)\cup (2|h|,\pi)} \frac{1}{|y|^{1-\alpha}} \ dy \leq C \norm{a}_{C^{2,\alpha}(\T)}\norm{b}_{C^{\alpha}(\T)} h^{\alpha}.
\end{align}
Therefore, collecting estimates \eqref{estimate:II}, \eqref{estimate:in} and \eqref{estimate:out} we have shown that
\begin{equation}
\abs{\mathsf{I}_{11}+\mathsf{I}_{12}}\leq C \norm{a}_{C^{2,\alpha}(\T)}\norm{b}_{C^{\alpha}(\T)} h^{\alpha}.
\end{equation}
Repeating the same estimates for $\mathsf{I}_{21}+\mathsf{I}_{22}$ one can readily check that
\begin{equation}
\abs{\mathsf{I}_{21}+\mathsf{I}_{22}}\leq C \norm{a}_{C^{1,\alpha}(\T)}\norm{b}_{C^{1,\alpha}(\T)} h^{\alpha}.
\end{equation}
Therefore, this concludes that 
\[ \abs{\Theta'(x+h)-\Theta'(x)}\leq C \norm{a}_{C^{2,\alpha}(\T)}\norm{b}_{C^{1,\alpha}(\T)} h^{\alpha},\]
and thus together with  \eqref{estimate:Linfty} we have shown that
\[ \norm{\Theta}_{C^{1,\alpha}(\T)}\leq \norm{a}_{C^{2,\alpha}(\T)}\norm{b}_{C^{1,\alpha}(\T)},\]
proving the desired result.
\end{proof}

To conclude this section, we recall the Crandall-Rabinowitz Theorem which is a fundamental tool in bifurcation theory that will be used to provide the main result of this article. To that purpose, let us first recall the following definition:
\begin{defi}[Fredholm operator]
Let $X$ and $Y$ be  two Banach spaces. A continuous linear mapping $T:X\rightarrow Y,$  is a  Fredholm operator if it fulfills the following properties,
\begin{enumerate}
\item $\textnormal{dim Ker}\,  T<\infty$,
\item $\textnormal{Im}\, T$ is closed in $Y$,
\item $\textnormal{codim Im}\,  T<\infty$.
\end{enumerate}
The integer $\textnormal{dim Ker}\, T-\textnormal{codim Im}\, T$ is called the Fredholm index of $T$.
\end{defi}

Next, we shall discuss  the index persistence through compact perturbations, cf. \cite{Kato, Kielhofer}.
\begin{pro}\label{prop:compact}
The index of a Fredholm operator remains unchanged under compact perturbations.
\end{pro}

Now, we recall the classical Crandall-Rabinowitz Theorem whose proof can be found  in \cite{Crandall-Rabi}.

\begin{theo}[Crandall-Rabinowitz Theorem]\label{CR}
    Let $X, Y$ be two Banach spaces, $V$ be a neighborhood of $0$ in $X$ and $F:\mathbb{R}\times V\rightarrow Y$ be a function with the properties,
    \begin{enumerate}
        \item $F(\lambda,0)=0$ for all $\lambda\in\mathbb{R}$.
        \item The partial derivatives  $\partial_\lambda F$, $\partial_fF$ and  $\partial_{\lambda}\partial_fF$ exist and are continuous.
        \item The operator $\partial_f F(\lambda_{0},0)$ is Fredholm of zero index and $\textnormal{Ker}(F_f(\lambda_{0},0))=\langle f_0\rangle$ is one-dimensional. 
                \item  Transversality assumption: $\partial_{\lambda}\partial_fF(\lambda_{0},0)f_0 \notin \textnormal{Im}(\partial_fF(\lambda_{0},0))$.
    \end{enumerate}
    If $Z$ is any complement of  $\textnormal{Ker}(\partial_fF(\lambda_{0},0))$ in $X$, then there is a neighborhood  $U$ of $(\lambda_{0},0)$ in $\mathbb{R}\times X$, an interval  $(-a,a)$, and two continuous functions $\Phi:(-a,a)\rightarrow\mathbb{R}$, $\beta:(-a,a)\rightarrow Z$ such that $\Phi(0)=\lambda_{0}$ and $\beta(0)=0$ and
    $$F^{-1}(0)\cap U=\{(\Phi(s), s f_0+s\beta(s)) : |s|<a\}\cup\{(\lambda,0): (\lambda,0)\in U\}.$$
\end{theo}
In this context, we will say that $\lambda_{0}$ is an eigenvalue of $F$.

\section{Formulation of the problem}\label{sec:3}

We shall look for traveling waves for $f$ and hence find $\varphi$ such that
$$
f(t,x)=\varphi(x-ct),
$$
for some speed $c\in\mathbb{R}$. Hence, the equation reduces to
$$
F[c,\varphi](\xi)=0, \quad x\in[-\pi,\pi],
$$
where
\begin{align}
F[c,\varphi](\xi)&=2c\varphi'(\xi) +c\alpha_0\Lambda[\varphi'](\xi)+\frac{1}{\varepsilon}\left\{\varphi'(\xi)+\mathcal{H}[\varphi](\xi)+(\alpha_0-\beta)\mathcal{H}[\varphi''](\xi)\right\}+\mathcal{H}[(\Lambda \varphi)^2](\xi)\nonumber \\
&\quad-\llbracket\mathcal{H},\varphi\rrbracket[\Lambda \varphi](\xi)+ (\alpha_0-\beta)\llbracket\mathcal{H},\varphi\rrbracket[\Lambda^3 \varphi](\xi)\nonumber \\
&=2c\varphi'(\xi) +(c\alpha_0+\frac{\alpha_0-\beta}{\varepsilon})\mathcal{H}[\varphi''](\xi)+\frac{1}{\varepsilon}\left\{\varphi'(\xi)+\mathcal{H}[\varphi](\xi)\right\}+\mathcal{H}[(\mathcal{H} \varphi')^2](\xi)\nonumber \\
&\quad \quad -\llbracket\mathcal{H},\varphi\rrbracket[\mathcal{H} \varphi'](\xi) -(\alpha_0-\beta)\llbracket\mathcal{H},\varphi\rrbracket[\mathcal{H} \varphi'''](\xi). \label{definicionF}
\end{align}
Hence, note that we have the following line of trivial solutions:
$$
F[c,0]=0, \quad \mbox{more generally } \ F[c,a]=0, \forall a\in \R.
$$

Define also the functional spaces
\begin{align*}
X:=\left\{f\in C^{3,\alpha}([0,2\pi],\R),\quad f(\xi)=\sum_{k\geq 1}f_k\cos(k\xi)\text{ with norm }\|f\|_{X}=\|f\|_{C^{3,\alpha}}\right\},\\
Y:=\left\{f\in C^{1,\alpha}([0,2\pi],\R),\quad f(\xi)=\sum_{k\geq 1}f_k\sin(k\xi)\text{ with norm }\|f\|_{Y}=\|f\|_{C^{1,\alpha}}\right\}.
\end{align*}

 \subsection{The linearized operator: spectral properties and transversality condition}\label{sec:31}
The first result shows that the operator $F$ defined in \label{definicionF} is well-defined and has the desired regularity. More precisely:
\begin{pro}\label{well:define1}
The operator $F:\mathbb{R}\times X\rightarrow Y$ given in \eqref{definicionF} is well-defined and $\mathscr{C}^1(\mathbb{R}\times X\rightarrow Y)$.
\end{pro}
\begin{proof}
Let us start checking that $F$ is well-defined. First of all, let us check the symmetry in the spaces. That is, if $\varphi(-\xi)=\varphi(\xi)$, then
$$
F[c,\varphi](-\xi)=-F[c,\varphi](\xi).
$$
Indeed, it is straightforward to check that
$$
\varphi'(-\xi)=-\varphi'(\xi), \quad \varphi''(-\xi)=\varphi''(\xi), \quad \varphi'''(-\xi)=-\varphi'''(\xi).
$$
Furthermore, note that
\begin{align*}
\mathcal{H}[\varphi''](-\xi)=\frac{1}{2\pi} p.v. \int_{-\pi}^\pi \frac{\varphi''(y)}{\tan\left(\frac{-\xi-y}{2}\right)}dy=-\frac{1}{2\pi}p.v. \int_{-\pi}^\pi \frac{\varphi''(-y)}{\tan\left(\frac{\xi-y}{2}\right)}dy=-\mathcal{H}[\varphi''](\xi).
\end{align*}
We can check, in a similar way, that the symmetry property is satisfied by the remaining integral terms. \medskip

Let us move next to the regularity properties for the operator $F$. We first notice that
\begin{align*}
\mathcal{H}[h](\xi)=&\sum_{k\geq 1}h_k\frac{1}{2\pi} p.v.\int_{-\pi}^{\pi}\frac{\cos(k y)}{\tan\left(\frac{\xi-y}{2}\right)}dy\\
=&-\sum_{k\geq 1}h_k\frac{1}{2\pi} p.v.\int_{-\pi}^{\pi}\frac{\cos(k (\xi-z))}{\tan\left(\frac{z}{2}\right)}dz\\
=&\sum_{k\geq 1}h_k\frac{1}{2\pi} p.v.\int_{-\pi}^{\pi}\frac{\sin(k\xi)\sin(kz)-\cos(k\xi)\cos(kz)}{\tan\left(\frac{z}{2}\right)}dz\\
=&\sum_{k\geq 1}h_k\sin(k\xi)\frac{1}{2\pi} p.v.\int_{-\pi}^{\pi}\frac{\sin(k y)}{\tan\left(\frac{y}{2}\right)}dy\\
=&\sum_{k\geq 1}h_k\sin(k\xi),
\end{align*}
and as a consequence we have in particular that $\mathcal{H}: X\to Y$. Furthermore, it is easy to check that the first three terms in \eqref{definicionF} are easily bounded by
\begin{align}
\| 2c\varphi'(\xi) +(c\alpha_0+\frac{\alpha_0-\beta}{\varepsilon})\mathcal{H}[\varphi''](\xi)+\frac{1}{\varepsilon}\left\{\varphi'(\xi)+\mathcal{H}[\varphi](\xi)\right\}) \|_{Y}\leq C \| \varphi \|_{X}.
\end{align}
Similarly, using the fact that  $\mathcal{H}: X\to Y$ and  the Banach Algebra property for $Y$ yields
\begin{align}
\| \mathcal{H}[(\mathcal{H} \varphi')^2](\xi) +\llbracket\mathcal{H},\varphi\rrbracket[\mathcal{H} \varphi'](\xi)\|_{Y}\leq C \| \varphi \|^{2}_{X}.   
\end{align}
In order to bound the commutator term  $(\alpha_0-\beta)\llbracket\mathcal{H},\varphi\rrbracket[\mathcal{H} \varphi'''](\xi)$ we make use of estimate \eqref{est:commutator} derived in Lemma \ref{lemma:comutator}. Indeed, taking $a=\varphi$ and $b=\mathcal{H} \varphi''$ in Lemma \ref{lemma:comutator} we have that
\begin{align}
 \norm{\llbracket\mathcal{H},\varphi\rrbracket[\mathcal{H} \varphi'''](\xi)}_{C^{1,\alpha}}&\leq C \norm{\varphi}_{C^{2,\alpha}}\norm{\mathcal{H}\varphi''}_{C^{1,\alpha}}\leq C\norm{\varphi}_{X}^{2}.
 \end{align}
Altogether, we have show as claimed that  $F:\mathbb{R}\times X\rightarrow Y$ given in \eqref{definicionF} is well-defined. Next, let us demonstrate that $\mathscr{C}^1(\mathbb{R}\times X\rightarrow Y)$. To do so, it is enough to observe that
\begin{equation}\label{estimacion:C1deF}
\|\partial_\varphi F[c,\varphi_1]h-\partial_\varphi F[c,\varphi_2]h\|_{Y}\leq C\|h\|_{X}\|\varphi_1-\varphi_2\|_{X}, 
\end{equation}
where 
\begin{align*}
\partial_\varphi F[c,\varphi]h&=2ch'+(c\alpha_0+\frac{\alpha_0-\beta}{\varepsilon})\mathcal{H}[h'']+\frac{1}{\varepsilon}\left\{h'+\mathcal{H}[h]\right\}+2\mathcal{H}[\mathcal{H} \varphi'\mathcal{H} h']-\llbracket\mathcal{H},h\rrbracket[\mathcal{H} \varphi']\\
&\quad-\llbracket\mathcal{H},\varphi\rrbracket[\mathcal{H} h']-(\alpha_0-\beta)\llbracket\mathcal{H},h\rrbracket[\mathcal{H} \varphi''']-(\alpha_0-\beta)\llbracket\mathcal{H},\varphi\rrbracket[\mathcal{H} h'''],
\end{align*}
denotes the Gateaux derivative. Indeed, using again the Banach Algebra property we find that
\[
\|2\mathcal{H}[\mathcal{H} (\varphi_{1}-\varphi_{2})'\mathcal{H} h']\|_{Y}\leq C \norm{\mathcal{H} (\varphi_{1}-\varphi_{2})'}_{Y}\norm{\mathcal{H} h'}_{Y}\leq C \norm{\varphi_{1}-\varphi_{2}}_{X}\norm{h}_{X} 
\]
Similarly, it easy to check that
\[
\|\llbracket\mathcal{H},h\rrbracket[\mathcal{H} (\varphi_{1}-\varphi_{2})']\|_{Y}+\| \llbracket\mathcal{H},(\varphi_{1}-\varphi_{2})\rrbracket[\mathcal{H} h']\|_{Y}\leq C\norm{\varphi_{1}-\varphi_{2}}_{X}\norm{h}_{X}.
\]
To conclude, we invoke Lemma \ref{lemma:comutator} with $a=h$ and  $b=\mathcal{H} (\varphi_{1}-\varphi_{2})''$ and $a=\varphi_{1}-\varphi_{2}$ and $b=\mathcal{H} h''$ respectively to find that
\[ \|\llbracket\mathcal{H},h\rrbracket[\mathcal{H} (\varphi_{1}-\varphi_{2})'''] \|_{Y}\leq C\norm{h}_{C^{2,\alpha}}\norm{\mathcal{H} (\varphi_{1}-\varphi_{2})''}_{C^{1,\alpha}}\leq C \norm{h}_{X}\norm{\varphi_{1}-\varphi_{2}}_{X},\]
\[\| \llbracket\mathcal{H},(\varphi_{1}-\varphi_{2})\rrbracket[\mathcal{H} h''']\|_{Y}\leq C \norm{\varphi_{1}-\varphi_{2}}_{C^{2,\alpha}}\norm{\mathcal{H} h''}_{C^{1,\alpha}}\leq C \norm{h}_{X}\norm{\varphi_{1}-\varphi_{2}}_{X},\]
which shows estimate \eqref{estimacion:C1deF}.
Hence, we can conclude that the Gateaux derivative is continuous (indeed, it is Liptschitz) and then we can ensure the existence and continuity of the Frechet derivative.
\end{proof}
In the following, we analyze the linearized operator at the trivial solution $(c,0)$ given by
\begin{align}\label{linear:F}
\partial_\varphi F[c,0]h(\xi)=2ch'(\xi) +(c\alpha_0+\frac{\alpha_0-\beta}{\varepsilon})\mathcal{H}[h''](\xi)+\frac{1}{\varepsilon}\left\{h'(\xi)+\mathcal{H}[h](\xi)\right\}.
\end{align}
More precisely, we study the Fredholm index of the operator \eqref{linear:F}.
\begin{pro}\label{prop:fredholm}
For $c\neq 0$, the operator $\partial_\varphi F[c,0]$ is Fredholm of zero index.
\end{pro}
\begin{proof}
Since the coefficient $c\alpha_0+\frac{\alpha_0-\beta}{\varepsilon}\neq 0$ we have that
\[ \partial_\varphi F[c,0]h(\xi)=\mathcal{L}h(\xi)+\mathcal{K}h(\xi),\]
where
\[ \mathcal{L}h(\xi)=(c\alpha_0+\frac{\alpha_0-\beta}{\varepsilon})\mathcal{H}[h''](\xi), \quad \mathcal{K}h(\xi)=2ch'+\frac{1}{\varepsilon}\left\{h'(\xi)+\mathcal{H}[h]\right\}.\]
The principal part of the linear operator $\mathcal{L}h$ is an isomorphism from $X$ to $Y$ and thus has zero index. Indeed, this follows by noticing that for $h\in X$ we have that 
\[ \mathcal{L}h(\xi)=-(c\alpha_0+\frac{\alpha_0-\beta}{\varepsilon})\sum_{k\geq 1}h_k k^2\sin(kx).\]
Moreover, for
\[ Z:=\left\{f\in C^{2,\alpha}([0,2\pi],\R),\quad f(\xi)=\sum_{k\geq 1}f_k\sin(k\xi)\text{ with norm }\|f\|_{Z}=\|f\|_{C^{2,\alpha}}\right\}, \]
the operator $\mathcal{K}h:X\to Z$ is continuous. Therefore, the embedding $Z\hookrightarrow Y$ is compact thus by Proposition \ref{prop:compact}, we conclude that \eqref{linear:F} is Fredholm of zero index.
 \end{proof}
 The following result describes the kernel and range of the linearized operator.
 \begin{pro}\label{prop:kernel}
If $h(x)=\displaystyle\sum_{k\geq 1}h_k \cos(kx)$, we have that
\begin{equation}\label{lin:inFourier}
\partial_\varphi F[c,0]h(x)=\sum_{k\geq 1}h_k \sin(kx)\left\{-(2c+\frac{1}{\varepsilon})k+\frac{1}{\varepsilon}-(c\alpha_0+\frac{\alpha_0-\beta}{\varepsilon}) k^2\right\}. 
\end{equation}
Hence, for
\[c_{k}= \frac{1}{\varepsilon}\left(\frac{1-k-(\alpha_{0}-\beta)k^2}{k(2+\alpha_{0}k)}\right),\quad k\geq 1\]
we have that the kernel and the range of the linearized operator can be described as follows
\begin{align*}
\textnormal{Ker}[\partial_{\varphi}F[c_k,0]]=<\cos(kx)>,\\
Y/\textnormal{Img}[\partial_{\varphi}F[c_k,0]]=<\sin(kx)>.
\end{align*}
Moreover, the transversal condition is satisfied, i.e. for  $h_0\in \textnormal{Ker}[\partial_{\varphi}F[c_k,0]]$, we find that
$$
\partial_c \partial_\varphi F[c_k,0]h_0\notin \textnormal{Im}[\partial_{\varphi}F[c_k,0]].
$$
 \end{pro}
 \begin{proof}
Let us first show how to obtain expression \eqref{lin:inFourier}. For $h(x)=\displaystyle\sum_{k\geq 1}h_k \cos(kx)$ we find that
\[
h'(\xi)= -\displaystyle\sum_{k\geq 1}h_k k\sin(kx), \ \mathcal{H}[h''](\xi)=-\sum_{k\geq 1}h_k k^2\sin(kx), \ \mathcal{H}[h](\xi)=\sum_{k\geq 1}h_k\sin(k\xi).
\]
Thus, recalling \eqref{linear:F} and the previous identities we infer that \eqref{lin:inFourier} holds. From the expression of the linearized operator in Fourier series \eqref{lin:inFourier}, it is clear that the kernel of $\partial_\varphi F[c_k,0]$ is generated by
$$
<\cos(kx)>.
$$
Moreover, since the linearized operator is Fredholm of zero index, one has that the codimension of the range is one dimensional and thus we can ensure that
$$
Y/\textnormal{Img}[\partial_{\varphi}F[c_k,0]]=<\sin(kx)>.
$$
Finally, to check the transversal condition we have to differentiate the linear operator with respect to the parameter $c$ obtaining
\[ \partial^2_{(\varphi,c)} F[c,0]h(x)=\sum_{k\geq 1}h_k \sin(kx)\left\{-2k-\alpha_0 k^2\right\}. \]
Next, we evaluate it at the generator of the kernel:
\[ \partial^2_{(\varphi,c)} F[c_{k_\star},0]\cos(k_\star x)= \sin(k_\star x)\left\{-2k_\star-\alpha_0 k_\star^2\right\}. \]
for $k_\star\geq 1$. Since $\alpha_0>0$ we find that
$$
\partial^2_{(\varphi,c)} F[c_{k_\star},0]\cos(k_\star x)\notin \textnormal{Img}[\partial_{\varphi}F[c_{k_\star},0]],
$$
and hence the transversal condition is satisfied.
 \end{proof}

%The linearized operator $\partial_\varphi F[c,0]:X\rightarrow Y$ is Fredholm of zero index. To see this is enough to observe that it is linear and bounded between the spaces $X$ and $Y$. \textcolor{orange}{In fact, 
%$$
%\partial_\varphi F[c,0]:X\rightarrow C^{1,\alpha}\subset \subset Y
%$$
%due to Ascoli-Arzela Theorem and as a consequence it is a compact operator.}
%
% Furthermore, its kernel is given by the constant functions and thus, it has dimension 1.
%%
%
%{\color{blue}(Claim:) }We can check that $\partial_\varphi F[c,0]$ can be written as a compact perturbation of an isomorphim.
%{\color{blue}(I don't know if this is true for $\alpha_0\neq \beta$, but this is not a problem, we just have to understand better the range of the lin operator)}
%
%

\subsection{Proof of Theorem \ref{main:theorem}}\label{sec:32}
Fix $m\geq 1$. In order to prove Theorem \ref{main:theorem}, let us introduce the symmetry $m$ in the spaces. For that, let us define
\begin{align*}
X_m:=\left\{f\in C^{3,\alpha}([0,2\pi],\R),\quad f(\xi)=\sum_{k\geq 1}f_k\cos(mk\xi)\text{ with norm }\|f\|_{X_m}=\|f\|_{C^{3,\alpha}}\right\},\\
Y_m:=\left\{f\in C^{1,\alpha}([0,2\pi],\R),\quad f(\xi)=\sum_{k\geq 1}f_k\sin(mk\xi)\text{ with norm }\|f\|_{Y_m}=\|f\|_{C^{1,\alpha}}\right\},
\end{align*}
for any $m\geq 1$. In order to check that
$$
F:\R\times X_m\rightarrow Y_m,
$$
is well-defined we can use Proposition \ref{well:define1} but it remains to check the m-fold symmetry property. For that purpose, we have to check that if 
$$
\varphi(\xi+\frac{2\pi}{m})=\varphi(\xi),
$$
then
$$
F[c,\varphi](\xi+\frac{2\pi}{m})=F[c,\varphi](\xi).
$$
Note that is $\varphi$ has the m-fold symmetry property, then all the derivatives also enjoy the same symmetry. Now, let us check the Hilbert term:
\begin{align*}
\mathcal{H}[\varphi''](\xi+2\pi/m)=\frac{1}{2\pi} p.v. \int_{-\pi}^\pi \frac{\varphi''(y)}{\tan\left(\frac{\xi-y+2\pi/m}{2}\right)}dy&=\frac{1}{2\pi}p.v. \int_{-\pi}^\pi \frac{\varphi''(y+2\pi/m)}{\tan\left(\frac{\xi-y+2\pi/m-2\pi/m}{2}\right)}dy \\
&=\mathcal{H}[\varphi''](\xi).
\end{align*}
Similar argument works for the other integral terms. Following Proposition \ref{prop:fredholm} the linear operator is a Fredholm operator of zero index, and Proposition \ref{prop:kernel} gives us the expression of the linear operator in Fourier series:
\[ \partial_\varphi F[c_m,0]h(x)=\sum_{k\geq 1}h_k \sin(mkx)\left\{-(2c+\frac{1}{\varepsilon})k+\frac{1}{\varepsilon}-(c\alpha_0+\frac{\alpha_0-\beta}{\varepsilon}) k^2\right\}. \]
Hence Proposition \ref{prop:kernel} gives us the one dimensionality of the kernel, which is now generated by $k=1$:
$$
<\cos(mx)>,
$$
as well as the one co-dimensionality of the range. Finally, the transversal condition is satisfied in Proposition \ref{prop:kernel}. Hence, Crandall-Rabinowitz theorem can be applied obtaining the main result of this paper.

\subsection*{Acknowledgements}D.A-O is supported by the fellowship of the Santander-ULL program. C.G. has been supported by RYC2022-035967-I (MCIU/AEI/10.13039 /501100011033 and FSE+), and partially by Grants PID2022-140494NA-I00 and PID2022-137228OB-I00 funded by \\
MCIN/AEI/10.13039/501100011033/FEDER, UE, by Grant C-EXP-265-UGR23 funded by Consejeria de Universidad, Investigacion e Innovacion \& ERDF/EU Andalusia Program, and by Modeling Nature Research Unit, project QUAL21-011. D.A-O and R. G-B are also supported by the project “An\'alisis Matem\'atico Aplicado y Ecuaciones Diferenciales” Grant PID2022-141187NB-I00 funded by MCIN/ AEI and acronym “AMAED”. R.G-B thanks the department of applied mathematics of the University of Granada where part of this research was performed for their hospitality.

\end{document}